\documentclass{proc-l-mod}

\makeatletter
\def\@printyear{TT}
\makeatother

\issueinfo{143}{04}{April}{2015}
\pagespan{1459}{1469}
\dateposted{November 12, 2014}
\PII{S 0002-9939(2014)12358-6}
\copyrightinfo{2014}{Henry Cohn and Abhinav Kumar}

\usepackage[pdftex,colorlinks,citecolor=black,linkcolor=black,urlcolor=black,bookmarks=false]{hyperref}
\def\MR#1#2{\href{http://www.ams.org/mathscinet-getitem?mr=#1}{MR#2}}
\def\doi#1{\href{http://dx.doi.org/#1}{DOI #1}}

\newcommand{\Z}{\mathbb{Z}}

\newcommand{\R}{{\mathbb R}}
\newcommand{\F}{{\mathbb F}}
\newcommand{\Ha}{{\mathbb H}}
\newcommand{\Hu}{{\mathcal H}}
\newcommand{\Proj}{{\mathbb P}}
\DeclareMathOperator{\tr}{tr}
\DeclareMathOperator{\N}{N}
\DeclareMathOperator{\ord}{ord}

\newcommand{\legendre}[2]{\genfrac(){}{}{#1}{#2}}

\theoremstyle{plain}
\newtheorem{theorem}{Theorem}
\newtheorem{corollary}[theorem]{Corollary}
\newtheorem{lemma}[theorem]{Lemma}
\newtheorem{proposition}[theorem]{Proposition}

\numberwithin{theorem}{section}

\theoremstyle{definition}

\numberwithin{equation}{section}

\title{Metacommutation of Hurwitz primes}

\author{Henry Cohn}
\address{Microsoft Research New England\\
One Memorial Drive\\
Cambridge, Massachusetts 02142}
\email{cohn@microsoft.com}

\author{Abhinav Kumar}
\address{Department of Mathematics\\
Massachusetts Institute of Technology\\
Cambridge, Massachusetts 02139}
\email{abhinav@math.mit.edu}

\thanks{The second author was supported in part by National Science Foundation grants
DMS-0757765 and DMS-0952486 and by a grant from the Solomon
Buchsbaum Research Fund.}

\subjclass[2010]{Primary 11R52, 11R27}

\commby{Matthew A.\ Papanikolas}

\date{December 31, 2012 and, in revised form, August 30, 2013.}

\begin{document}

\begin{abstract}
Conway and Smith introduced the operation of {\em
metacommutation} for pairs of primes in the ring of Hurwitz
integers in the quaternions. We study the permutation induced
on the primes of norm $p$ by a prime of norm $q$ under
metacommutation, where $p$ and $q$ are distinct rational
primes. In particular, we show that the sign of this
permutation is the quadratic character of $q$ modulo $p$.
\end{abstract}

\maketitle

\section{Introduction}

In this paper we study the metacommutation mapping, which we
will define shortly.  Let $\Ha = \R + \R i + \R j + \R k$ be
the algebra of Hamilton quaternions over $\R$, and let $\Hu$ be
the subring of Hurwitz integers, consisting of the quaternions
$a + bi + cj + dk$ for which $a,b,c,d$ are all in $\Z$ or all
in $\Z + 1/2$. Recall that for an element $x = a + bi + cj +
dk$, its \emph{conjugate} $x^\sigma$ is $a - bi - cj - dk$ (we
reserve the notation $\bar x$ for reduction modulo a prime).
The reduced\footnote{Recall that the norm and trace are
``reduced'' because the norm and trace of the linear
transformation of multiplication by $x$ on $\Ha$ are $\N(x)^2$
and $2 \tr(x)$, respectively, for either left or right
multiplication.} \emph{norm} of $x$ is $\N(x) = xx^\sigma = a^2
+ b^2 + c^2 + d^2$, and the reduced \emph{trace} of $x$ is
$\tr(x) = x + x^\sigma = 2a$.

It is well known that $\Hu$ has a Euclidean division algorithm
(on the left or right), and thus every one-sided ideal is
principal. An element $P \in \Hu$ is \emph{prime} if it is
irreducible; i.e., it is not a product of two nonunits in
$\Hu$. We say a left prime ideal $\mathfrak{P} = \Hu P$
\emph{lies over} the rational prime $p$ if $p \in
\mathfrak{P}$, which is equivalent to $\N(P) = p$. If $p$ is
odd, there are exactly $p+1$ Hurwitz primes lying over $p$, up
to left multiplication by units (in other words, there are
exactly $p+1$ prime left ideals lying over $p$), while there is
only one Hurwitz prime over $p = 2$.

Because of the Euclidean algorithm for $\Hu$, we may factor any
nonzero Hurwitz integer $x$ as
\[
x = P_1 P_2 \dots P_n,
\]
where $P_i$ is a Hurwitz prime of norm $p_i$. We say that this is a
factorization of~$x$ \emph{modeled on} a factorization of $\N(x)$ as
$p_1 \dots p_n$ (where we pay attention to the order of the rational
primes $p_1,\dots,p_n$). When $\N(x)$ is square-free, such a
factorization is unique up to \emph{unit migration}. That is, if
$p_1,\dots,p_n$ are distinct, then every
\newpage 
\noindent 
factorization of $x$ modeled
on $\N(x) = p_1 \dots p_n$ is of the form
\[
x = (P_1 u_1) (u_1^{-1} P_2 u_2) \dots (u_{n-1}^{-1} P_n),
\]
where $u_1, \dots, u_{n-1}$ are units of $\Hu$.

However, if we want to look at all possible prime factorizations of $x$, then
we must allow for changes in the order of the primes $p_1,\dots,p_n$.  If $P$
and $Q$ are primes lying over distinct rational primes $p$ and $q$, then $PQ$
has a unique (up to unit migration) factorization $Q' P'$ modeled on $qp$.
This process of switching two adjacent primes in the model was named
\emph{metacommutation} by Conway and Smith in their account of unique
factorization for Hurwitz integers (see Chapter~5 of \cite{CS}). They prove
that the prime factorization of an arbitrary nonzero Hurwitz integer is
unique up to metacommutation, unit migration, and \emph{recombination}: the
process of replacing $P P^\sigma$ with $\tilde P \tilde P^\sigma$, where $P$
and $\tilde P$ are primes of the same norm.  Another exposition of this
theorem can be found in \cite{CP}.

Conway and Smith comment that the metacommutation problem of
determining $Q'$ and $P'$ given $P$ and $Q$ does not seem to be
addressed in the literature \cite[p.~61]{CS}.  Their proof that $Q'$
and $P'$ exist yields an efficient method for computing them via the
Euclidean algorithm, but it provides little insight into the properties
of the metacommutation map.

In this paper, we analyze the metacommutation map.  Let $p$ be a rational
prime and $Q$ a Hurwitz prime of norm $q \ne p$.  For every prime $P$ of norm
$p$, we can find $Q'$ and $P'$ of norms $q$ and $p$ satisfying $PQ=Q'P'$, and
the pair $(Q',P')$ is unique up to unit migration. Thus, $P'$ is defined only
up to left multiplication by a unit, and the equation
\[
(uP)Q = (uQ') P'
\]
implies that replacing $P$ by a left associate has no effect on $P'$.
Therefore, metacommutation by $Q$ yields a permutation of the $p+1$ primes
lying above $p$.  We determine the number of fixed points (i.e., primes $P$
that commute with $Q$ modulo units) and sign of this permutation:

\begin{theorem} \label{theorem:main}
Let $p$ and $q$ be distinct rational primes, let $Q$ be a Hurwitz
prime of norm $q$, and consider the Hurwitz primes of norm $p$
modulo left multiplication by units.  Metacommutation by $Q$ permutes
these primes, and the sign of the permutation is the quadratic
character $\legendre{q}{p}$ of $q$ modulo $p$.

If $p = 2$, or if $Q$ is congruent to a rational integer modulo $p$, then
metacommutation by $Q$ is the identity permutation, and otherwise it
has $1+\legendre{\tr(Q)^2-q}{p}$ fixed points.
\end{theorem}

The simplest reason why $p$ and $q$ could determine the sign of the
permutation would be if they determined its cycle structure. However, the
fixed point count shows that this is not true.  For example, the
metacommutation permutations of the primes of norm $3$ by $1+2i$ and $i+2j$
have $0$ and $2$ fixed points, respectively.

Note that when $P$ is a fixed point under metacommutation by $Q$, it is both
a left and a right divisor of $PQ=Q'P$. See \cite{AADGS} for an analysis of
which Hurwitz integers have common left/right divisors and how many they
have.

Theorem~\ref{theorem:main} is trivial to check when $p = 2$, so we shall
assume from now on that $p$ is odd.

\section{A conic associated to the primes above $p$}

Let $P$ be a Hurwitz prime of odd norm $p$, and let $\Pi$ be
the reduction $\overline{\Hu P}$ of the left ideal $\Hu P$
modulo $p$; it is a left ideal of the quaternion algebra
$\overline{\Hu} = \Hu/p \Hu$ over $\F_p$. This quaternion
algebra is split, since its norm form, being a nondegenerate
quadratic form in four variables, has a nontrivial zero.
(Recall that it is a standard exercise to show, using the
pigeonhole principle, that for $\alpha,\beta \in \F_p \setminus
\{0\}$, the equation $\alpha x^2 + \beta y^2 = 1$ has a
solution $(x,y) \in \F_p^2$.)

Therefore $\overline{\Hu} \cong M_2(\F_p)$, and since every
finite-dimensional module over the matrix algebra $M_2(\F_p)$ is a
direct sum of standard modules $\F_p^2$, we see that $\Pi$ must be
even-dimensional. It is not the zero ideal since it contains
$\overline{P}$, and it is not the entire algebra since it is
annihilated by right multiplication by $\overline{P^\sigma}$ (because
$P P^\sigma = p$).  Thus, $\Pi$ must be two-dimensional inside the
four-dimensional algebra $\overline{\Hu}$.

In fact, prime ideals in $\Hu$ lying over $p$ are in one-to-one
correspondence with two-dimensional left ideals of $\overline{\Hu}$.
Specifically, taking the inverse image under the reduction map gives a
bijection between ideals in $\overline{\Hu}$ and ideals in $\Hu$ that
contain $p$.  Ideals in $\Hu$ containing $p$ are generated by divisors
of $p$, which are $1$, $p$, and the primes of norm $p$, and only the
latter correspond to two-dimensional ideals in $\overline{\Hu}$.  Thus,
studying the primes of norm $p$ is equivalent to studying the proper
ideals in $\overline{\Hu}$.

\begin{lemma} \label{tracezero}
There is a unique nonzero element $t = xi + yj + zk$ of trace zero in
$\Pi$, up to scaling.
\end{lemma}

\begin{proof}
The reduced trace gives a linear functional on the two-dimensional
space $\Pi$. Therefore its kernel $K = \ker(\tr)$ is at least
one-dimensional. Furthermore, $K$ cannot be all of $\Pi$, since if $xi
+ yj + zk \in K \setminus\{0\}$, then at least one of $x,y,z$ is not
zero, say $x$. Then multiplying on the left by $i$, we get $-x + yk -
zj$, which is in $\Pi$ (because $\Pi$ is a left ideal in
$\overline{\Hu}$), but not in $K$. Therefore $K$ has dimension $1$.
\end{proof}

Next, we observe that for every $t \in \Pi$, the reduced norm $\N(t)
\in \F_p$ must be zero, since $t = \overline{h P}$ for some $h \in \Hu$
and $\N(hP) = \N(h) \N(P) = p \N(h)$, which vanishes modulo $p$.

Let $C_p \subset \Proj^2(\F_p)$ be the conic defined by $x^2 + y^2 +
z^2 = 0$. Given a prime $P$ of norm $p$, let $t_P$ be the element of
$\Pi$ from Lemma~\ref{tracezero} (scaled arbitrarily), and let $c_P$ be
the corresponding point on the conic $C_p$.

\begin{proposition}
The map $P \mapsto c_P$ is a bijection between primes lying over $p$ up
to left multiplication by units and points on the conic $C_p$.
\end{proposition}

\begin{proof}
Because primes lying over $p$ correspond to two-dimensional ideals of
$\overline{\Hu}$, it suffices to deal with those ideals. Given a point
$c \in C_p$ with homogeneous coordinates $(x,y,z)$, let $\Pi_c =
\overline{\Hu}(xi+yj+zk)$. It is a nonzero left ideal and is not the
full ring $\overline{\mathcal H}$, because every element $u \in \Pi_c$
satisfies $u (xi+yj+zk) = 0$. This describes the inverse map and
establishes the bijection.
\end{proof}

Since the conic $C_p$ is smooth, it has exactly $p+1$ points. We
therefore recover the well-known count of the number of primes above
$p$.

\begin{corollary}
Let $p$ be an odd rational prime.  Then the number of Hurwitz primes
above $p$, up to left multiplication by units, is $p+1$.
\end{corollary}

The proof given here is fundamentally the same as that in
\cite{Pe}. In Vorlesung~9 of \cite{H}, Hurwitz develops a
similar method to count left ideals in $\overline{\Hu}$,
although he does not phrase it in terms of the conic $C_p$.
From a more abstract perspective, $C_p$ is a Severi-Brauer
variety (see \S 6 of Chapter~X in \cite{Se}). For a
generalization to\break 
Eichler orders in quaternion algebras over
number fields, we refer the reader to Th\'eor\`eme~II.2.3
in~\cite{Vi} or Lemma~7.2 in~\cite{KV}.

\section{Metacommutation}

Suppose $P$ and $Q$ are Hurwitz primes lying over distinct rational
primes $p$ and $q$, respectively.  By metacommutation, there are primes
$Q'$ and $P'$ of norms $q$ and $p$ such that
\[
PQ = Q' P',
\]
and we would like to understand how they depend on $P$ and $Q$.

It is not difficult to compute $P'$ and $Q'$ efficiently as follows.
Using the Euclidean algorithm, choose $P' \in \Hu$ so that $\Hu p + \Hu
PQ = \Hu P'$. Then $P'$ divides $p$, so its norm is $1$, $p$, or $p^2$.
It cannot be $1$ since every element of $\Hu p + \Hu PQ$ has norm
divisible by $p$, and it cannot be $p^2$ since $PQ \not\in \Hu p$.
Thus, $P'$ has norm $p$. Now $PQ \in \Hu p + \Hu PQ = \Hu P'$, so there
exists $Q'$ such that $PQ = Q'P'$, and $Q'$ must have norm $q$.

As explained in the introduction, $P'$ is defined only up to left
multiplication by a unit, and replacing $P$ by a left associate has no effect
on $P'$. Therefore, metacommutation by any prime $Q$ of norm $q \neq p$ gives
a permutation $\tau_Q$ of the $p+1$ primes lying above $p$. In the rest of
the paper, we analyze this action.

Note that replacing $Q$ by a right associate $Qu$ composes the
permutation $\tau_Q$ with the permutation $\gamma_u$ sending $P$ to
$Pu$, because
\[
PQu = Q'P'u = Q'(P'u),
\]
which implies $\tau_{Qu} = \gamma_u \circ \tau_Q$. Similarly,
\[
(Pu^{-1})(uQ) = PQ = Q'P'
\]
implies $\tau_{uQ} \circ \gamma_{u^{-1}} = \tau_Q$, or $\tau_{uQ} = \tau_Q \circ \gamma^{-1}_{u^{-1}} = \tau_Q \circ \gamma_u$.

To understand metacommutation more clearly, it is important to rephrase
it in terms of the constructions from the previous section.  We will
continue to use the notation from that section ($\Pi$, $t_P$, etc., and
the corresponding versions $\Pi'$, $t_{P'}$, etc.\ for $P'$).

\begin{theorem}
For $P'$ and $Q'$ as above, $t_{P'}$ and $(\overline{Q})^{-1} t_P
\overline{Q}$ are equal modulo scaling by $\F_p^\times$.
\end{theorem}

\begin{proof}
Because $PQ=Q'P'$, we see that $\overline{P} \overline{Q}$ is contained
in $\Pi'$, which is a left ideal. Therefore any left multiple of
$\overline{P} \overline{Q}$ is also in $\Pi'$, and hence $\Pi
\overline{Q} \subseteq \Pi'$.  In fact, equality holds, since
$\overline{Q}$ is invertible in $\overline{\Hu}$ and thus both $\Pi
\overline{Q}$ and $\Pi'$ are two-dimensional vector spaces over $\F_p$.
Finally, since $\overline{Q}$ is invertible and these are left ideals,
\[
\overline{Q}^{-1} \Pi \overline{Q} = \Pi \overline{Q} = \Pi'.
\]

In particular, since $t_P \in \Pi$, we must have $(\overline{Q})^{-1}
t_P \overline{Q} \in \Pi'$. However, the vector $(\overline{Q})^{-1}
t_P \overline{Q}$ has trace zero and is nonzero, and therefore it must
equal $t_{P'}$ up to scaling.
\end{proof}

Now suppose that $t_P = xi+yj+zk$ and $\overline{Q} =
a+bi+cj+dk$ with $\bar{q} = a^2 + b^2 + c^2 + d^2$. Then
explicit calculation shows that $(\overline{Q})^{-1} t_P
\overline{Q} = x'i+y'j+z'k$ with
\[
\begin{pmatrix}
x'\\
y'\\
z'
\end{pmatrix}
= \frac{1}{\bar{q}}
\begin{pmatrix}
a^2 + b^2 - c^2 - d^2 & 2ad + 2bc & -2ac + 2bd \\
-2ad + 2bc & a^2 + c^2 - b^2 - d^2 & 2ab + 2cd \\
2ac + 2bd & -2ab + 2cd & a^2 + d^2 - b^2 - c^2
\end{pmatrix}
\begin{pmatrix}
x\\
y\\
z
\end{pmatrix}
.
\]
(Note that the above matrix is the standard parametrization of
$SO(3)$ by $S^3$ via the double cover $SU(2) \to SO(3)$.) Let
us denote this linear transformation by $\phi_Q$. Note that
$\phi_Q$ lies in $SO_3(\F_p)$.  This is not hard to see
conceptually: $\phi_Q \in O_3(\F_p)$ because conjugation
preserves the quadratic form $v v^\sigma$ on quaternions $v$,
and $\det(\phi_Q)=1$ because left or right multiplication by a
quaternion $v$ has determinant $\N(v)^2$ on $\overline{\Hu}$
(so conjugation has determinant $1$ as a linear transformation
on $\overline{\Hu}$ and hence also on the purely imaginary
subspace, since the identity element of $\overline{\Hu}$ is
fixed).

The vector $(b,c,d)$ is an eigenvector of $\phi_Q$ with eigenvalue $1$
(unless it is the zero vector, in which case $\phi_Q$ is the identity
transformation), since $bi+cj+dk$ obviously commutes with
$\overline{Q}$. We are especially interested in eigenvectors that lie
on the cone $K_p$ given by $x^2 + y^2 + z^2 = 0$, as these correspond
to fixed points of the permutation $\tau_Q$. Note that the eigenvector
$(b,c,d)$ lies on the cone $K_p$ iff $\bar{q} = a^2$.

\begin{lemma} \label{lemma:charpoly}
The characteristic polynomial of $\phi_Q$ is
\[
(x-1)\left(x^2  + 2\left(1 - \frac{2a^2}{\bar{q}}\right) x + 1 \right) .
\]
\end{lemma}

It follows that the eigenvalue $1$ has multiplicity greater than $1$
iff $\bar{q} = a^2$, and the eigenvalue $-1$ occurs iff $a = 0$, in
which case it occurs with multiplicity $2$.  The simplest way to prove
Lemma~\ref{lemma:charpoly} is by direct computation using the matrix
for $\phi_Q$, but that is cumbersome.  As an alternative, we give a
coordinate-free proof.

\begin{proof}
We begin with the observations that $1$ is always an eigenvalue and
that $\det(\phi_Q)=1$.  To determine the characteristic polynomial, all
we need to show is that
\[
\tr(\phi_Q) = \frac{4a^2}{\bar{q}} - 1.
\]
Equivalently, on the four-dimensional space $\overline{\Hu}$, the map
$v \mapsto (\overline{Q})^{-1} v \overline{Q}$ should have trace
$4a^2/\bar{q}$, because the identity element of $\overline{\Hu}$ is an
eigenvector with eigenvalue $1$.  Let $u=bi+cj+dk$ be the imaginary
part of $\overline{Q}$, so $\overline{Q} = a+u$ and
$(\overline{Q})^{-1} = (a-u)/\bar{q}$.  Then we must show that the
trace of $v \mapsto (a+u)v(a-u)$ on $\overline{\Hu}$ is $4a^2$.  We can
write
\[
(a+u)v(a-u) = a^2 v + a(uv - vu) - uvu.
\]
The map $v \mapsto a^2 v$ has trace $4a^2$, and the maps $v \mapsto uv$
and $v \mapsto vu$ have trace $2\tr(u)=0$.  Thus, all that remains is
to show that $v \mapsto uvu$ has trace $0$.  Call this map $\varphi$.
We will construct an invertible map $\psi$ that anticommutes with
$\varphi$, so that $\psi \varphi \psi^{-1} = - \varphi$ and hence
$\tr(\varphi)=0$.  Specifically, we set $\psi(v) = u' v$, where
$u'$ is a nonzero imaginary quaternion that anticommutes with $u$.  To
see that such a $u'$ exists, note that $v \mapsto uv+vu$ sends
imaginary quaternions $v$ to scalars, since it is easy to check that
$(uv+vu)^\sigma = uv+vu$.  Thus, it has a nontrivial kernel, so $u'$
exists and $\tr \varphi = 0$, as desired.
\end{proof}

\section{Fixed points}
\label{sec:fixed}

We now discuss the fixed points of the permutation $\phi_Q$, considered
as a projective transformation. In particular, we will be most
interested in those fixed points that lie on the conic $C_p$ defined by
$x^2 + y^2 + z^2 = 0$. In other words, we are interested in
eigenvectors of $\phi_Q$ that lie on the cone $K_p \subset \F_p^3$
given by the same equation.

\begin{lemma} \label{easylem}
Let $v$ be an eigenvector of $\phi_Q$ of eigenvalue $\lambda$. If
$\lambda \neq \pm 1$, then $v$ lies on the cone $K_p$.
\end{lemma}

\begin{proof}
In terms of the inner product defined by the quadratic form $\langle
v,v \rangle = v v^\sigma$, we have
\[
\lambda^2 \langle v,v \rangle = \langle \lambda v, \lambda v \rangle
= \langle \phi_Q (v), \phi_Q (v) \rangle = \langle v,v \rangle,
\]
with the last equality because $\phi_Q$ is an orthogonal
transformation.
\end{proof}

Thus, the study of the fixed points falls naturally into the following
cases.

\medskip
\noindent \emph{Case} 1A: $b = c = d = 0$.
In this case, $\phi_Q$ is the identity transformation.

\medskip
\noindent \emph{Case} 1B: $\bar{q} = a^2$ but $(b,c,d) \neq (0,0,0)$.
The characteristic polynomial of $\phi_Q$ is $(x-1)^3$; i.e., the
eigenvalue $1$ occurs with multiplicity $3$. Also, the vector $(b,c,d)$
is on the cone $K_p$, since
\[
a^2 = \bar{q} = a^2 + b^2 + c^2 + d^2.
\]
Note that $v_0 = bi + cj + dk$ commutes with $\overline{Q}$. We
claim that $\phi_Q$ has no other eigenvectors up to scaling.
Because the eigenvalue must be $1$, an eigenvector corresponds
to a purely imaginary quaternion that commutes with
$\overline{Q}$ and hence also with $v_0$. If $x$ were another
eigenvector linearly independent of $v_0$, then $1$, $x$, and
$v_0$ would generate an abelian subalgebra of $\overline{\Hu}$
of dimension at least three, which is impossible.

\medskip
\noindent \emph{Case} 2: $a = 0$.
Then $b^2 + c^2 + d^2 = \bar{q} \neq 0$, so the eigenvector
$(b,c,d)$ of eigenvalue $1$ is not on the cone $K_p$. The
eigenvalue $-1$ occurs with multiplicity $2$, so it is not
immediately clear that $\phi_Q$ is semisimple. We can see that
the $(-1)$-eigenspace is two-dimensional, as follows: the
equation $\overline{Q}x = -x\overline{Q}$ is equivalent to
$\tr(x\overline{Q}) = 0$, since $x$ and $\overline{Q}$ are
purely imaginary. The linear functional $x \mapsto
\tr(\overline{Q}x)$ has at least a two-dimensional kernel in
the three-dimensional space of imaginary quaternions, which
proves our assertion.

It still remains to be seen whether these eigenvectors are in
the cone $K_p$.  The points in the cone $K_p$ and the
$(-1)$-eigenspace are given by the simultaneous solutions of
the two equations
\begin{align*}
bx + cy + dz &= 0, \\
x^2 + y^2 + z^2 &= 0.
\end{align*}
In the projective plane, this is the intersection of a line and a
conic, so there are two solutions over the algebraic closure
$\overline{\F}_p$ if we count with multiplicity. Over the field $\F_p$
itself, there are either zero or two solutions: the number of solutions
is not one, because the line is not tangent to the conic (it would be
tangent iff $b^2+c^2+d^2=0$).

In fact, a simple discriminant calculation shows that the
solutions are defined over $\F_p$ if and only if
$-(b^2+c^2+d^2)$ is a square in $\F_p$.  Thus, there are two
distinct points of intersection when $\legendre{-q}{p}=1$ and
none when $\legendre{-q}{p} = -1$, because
$b^2+c^2+d^2=\bar{q}$.

\medskip
\noindent \emph{Case} 3: $a \neq 0$ and $\bar{q} \neq a^2$. The
latter inequality implies $b^2 + c^2 + d^2 \neq 0$, and so $(b,c,d)$
is not on the cone $K_p$. The quadratic part of the characteristic
polynomial,
\[
x^2  + 2\left(1 - \frac{2a^2}{\bar{q}}\right) x + 1 ,
\]
has two distinct roots in $\overline{\F}_p$, not equal to $\pm
1$. The discriminant of this polynomial is $-16a^2
(\bar{q}-a^2)/\bar{q}^2$, and so there are two eigenvalues in
$\F_p$ if $\legendre{a^2-q}{p} = 1$ and none if
$\legendre{a^2-q}{p} = -1$. Note that the corresponding
eigenvectors are distinct from each other and from $(b,c,d)$,
since they belong to distinct eigenvalues not equal to $\pm 1$.
Thus, there are either two or zero fixed points on the cone
according to the quadratic character of $a^2 - \bar{q}$, by
Lemma~\ref{easylem}.

We sum up these observations in the following theorem, which
unifies the different cases and is the second part of
Theorem~\ref{theorem:main}.

\begin{theorem}
Let $Q$ be a prime of norm $q$, with $\overline{Q} = a + bi + cj + dk
\in \overline{\Hu}$. If $b=c=d=0$, then $\tau_Q$ is the identity
permutation.  Otherwise, $\tau_Q$ has $1 + \legendre{a^2-q}{p}$ fixed
points.
\end{theorem}

\section{Special orthogonal groups and conics over
finite fields}

Let $p$ be an odd prime and $F$ a nondegenerate binary
quadratic form over $\F_p$. In this section, we gather some
general results about $SO(F)$ and its action on affine plane
conics given by $F(x,y) = u$. We begin by recalling a few
well-known results.

\begin{lemma}
Every nondegenerate binary quadratic form $F(x,y)$ is
equivalent over $\F_p$ to $x^2 - ty^2$ for some $t \in \F_p$.
\end{lemma}

\begin{proof}
First, note that since $p$ is odd, we may diagonalize the form,
so we may assume $F(x,y) = \alpha x^2 + \beta y^2$ for some
$\alpha, \beta \in \F_p$. Since $\alpha$ and $\beta$ are both
nonzero, the form represents $1$ by the pigeonhole principle.
Since it has discriminant $\alpha\beta$, it must be equivalent
to the diagonal form $x^2 + \alpha\beta y^2$.
\end{proof}

From now on, we work with the diagonal form $g_t(x,y) = x^2 - ty^2$,
where $t \in \F_p \setminus \{ 0\}$.

\begin{lemma} \label{lemma:affconic}
Let $u \in \F_p \setminus \{0\}$. The number of points on the affine
conic $x^2 - ty^2 = u$ is $p - \legendre{t}{p}$.
\end{lemma}

\begin{proof}
The projective conic has $p+1$ points, since it is isomorphic to
$\Proj^1(\F_p)$, and the number of points at infinity is either $2$ or
$0$, according as whether $t$ is or is not a square modulo $p$.
\end{proof}

\begin{proposition} \label{prop:cyclic}
The group $SO(g_t)$ is cyclic of order $p - \legendre{t}{p}$. Every
element of $SO(g_t)$ is semisimple.
\end{proposition}

\begin{proof}
The group $O(g_t)$ consists of $2 \times 2$ matrices whose
columns are orthonormal with respect to $g_t$.  If the first
column is $(\alpha,\beta)$, with $\alpha^2-t\beta^2=1$, then
the second must be proportional to $(\beta t,\alpha)$ and the
constant of proportionality must be $\pm 1$.  For the matrix to
be in $SO(g_t)$, that constant must be $1$.  Thus, the elements
of $SO(g_t)$ are the matrices of the form
\[
M_{\alpha,\beta} =
\begin{pmatrix}
\alpha & \beta t \\
\beta & \alpha
\end{pmatrix}
\]
with $\alpha^2 - t \beta^2 = 1$. By Lemma~\ref{lemma:affconic},
the size of the group is as claimed.

Multiplication of matrices in $SO(g_t)$ leads to the group law
\[
M_{\alpha,\beta}M_{\gamma,\delta} = M_{\alpha\gamma + \beta\delta t, \alpha\delta + \beta\gamma}.
\]

If $t = r^2$ is a square, then $\alpha^2 - \beta^2 t = \alpha^2
- \beta^2 r^2 = (\alpha-\beta r)(\alpha + \beta r)$. It is easy
to see that the map $\eta\colon SO(g_t) \to \F_p^\times$ given by
$\eta(M_{\alpha,\beta}) = \alpha + \beta r$ is an isomorphism
of groups. Since $\F_p^\times$ is cyclic of order $p-1$, so is
$SO(g_t)$.

If $t$ is not a square, consider the map $\eta\colon SO(g_t) \to
\F_{p^2}^\times$ given by $\eta(M_{\alpha,\beta}) = \alpha +
\beta\sqrt{t}$. It gives an isomorphism of $SO(g_t)$ with the
subgroup $N$ of elements of $\F_{p^2}^\times$ of norm $1$.
These are the solutions of $z^{p+1} = 1$ in $\F_{p^2}^\times$,
so they are $p+1$ in number. Since $N$ is a finite subgroup of
the multiplicative group of a field, it is also cyclic.

Note that in either case, we mapped $M_{\alpha,\beta}$ to one
of its eigenvalues (either $\alpha + \beta r$ or $\alpha +
\beta \sqrt{t}$). In fact, in the first case $r$ is a square
root of $t$, so both maps are essentially the same.

To show that every element of $SO(g_t)$ is semisimple, observe
that the characteristic polynomial of $M_{\alpha,\beta}$ is
\[
x^2 - 2\alpha x + 1,
\]
which has discriminant $4(\alpha^2 -1)$. If $\alpha \neq \pm
1$, then the eigenvalues in $\F_{p^2}$ are distinct, so
$M_{\alpha,\beta}$ is semisimple. If $\alpha = \pm 1$, then
$\alpha^2 - t \beta^2 = 1$ forces $\beta = 0$, and therefore
$M_{\alpha,\beta}$ is either the identity or its negation.
Alternatively, one may use the fact that $SO(g_t)$ is a group of order
$p \pm 1$, which is coprime to the characteristic, to deduce
semisimplicity of its elements.
\end{proof}

\begin{proposition} \label{prop:quadchar}
Let $\psi \in SO(g_t)$ have characteristic polynomial $x^2 + (2-v) x +
1 = (x+1)^2 - vx$, with $v \neq 0$. Then $v$ is a square in $\F_p$ iff
\[
\frac{| SO(g_t) | }{\ord(\psi)}
\]
is even.
\end{proposition}

\begin{proof}
First, note that since $\psi$ is semisimple, its order is equal to the
order of either of its two eigenvalues, which are reciprocals of each
other. Therefore $| SO(g_t) |/\ord(\psi)$ is even iff an eigenvalue
$\lambda$ is the square of an element in the corresponding cyclic group
$\F_p^\times$ or $N$, with the notation from the proof of the previous
proposition.

First, suppose that $t = r^2$ is a square, so $SO(g_t) \cong
\F_p^\times$. Because the characteristic polynomial is $(x+1)^2-vx$, we
have
\[
v = \frac{ (\lambda + 1)^2 }{\lambda},
\]
which implies $v$ is a square in $\F_p^{\times}$ iff $\lambda$ is as well.

Next, suppose $t$ is not a square, so $SO(g_t) \cong N$. If $\lambda =
\mu^2$ with $\mu \in N$, then
\[
v = \lambda + \frac{1}{\lambda} + 2 = \mu^2 + \frac{1}{\mu^2} + 2 = \left(\mu + \frac{1}{\mu} \right)^2
\]
and $\mu + \mu^{-1} = \mu + \mu^p = \tr(\mu) \in \F_p$. Therefore $v$
is a square in $\F_p^\times$.  Conversely, if $v = w^2$ is a square
then let $\mu$ be a square root of $\lambda$ in $\overline{\F}_p$. We
need to show that $\mu \in N$. The above calculation shows that
\[
\mu + \frac{1}{\mu} = \pm w
\]
and therefore $\mu$ satisfies a quadratic equation
\[
\mu^2  \mp w \mu + 1 = 0.
\]
Thus $\mu \in \F_{p^2}$ and $\mu$ has norm $1$.  It follows that $\mu
\in N$, so $\lambda$ is a square in $N$.
\end{proof}

\begin{lemma}
Let $u \neq 0$, and let $D_{t,u}$ be the affine conic given by $x^2 - t
y^2 = u$. Then $SO(g_t)$ acts simply transitively on $D_{t,u}$.
\end{lemma}

\begin{proof}
The characteristic polynomial of $M_{\alpha,\beta}$ in the
variable $x$ is $x^2 - 2\alpha x + 1$, so $1$ is an eigenvalue
if and only if $\alpha=1$, in which case $\alpha^2-t\beta^2=1$
implies $\beta=0$ and hence $M_{\alpha,\beta}$ is the identity
matrix. Thus, no element of $SO(g_t)$ except the identity
matrix can fix any nonzero vector.  Because $SO(g_t)$ preserves
$D_{t,u}$ and $|SO(g_t)| = |D_{t,u}|$, it must act simply
transitively.
\end{proof}

\begin{corollary} \label{cor:sign}
Let $\psi \in SO(g_t)$ have characteristic polynomial $x^2 + (2-v) x +
1 = (x+1)^2 - vx$, with $v \neq 0$. Then the sign of the permutation
induced by $\psi$ on $D_{t,u}$ is $\legendre{v}{p}$.
\end{corollary}

\begin{proof}
Let $k = \ord(\psi)$ and $\ell = |SO(g_t)|/k$. Since $SO(g_t)$ is
cyclic and its action on $D_{t,u}$ is simply transitive, the
permutation induced on $D_{t,u}$ by any generator of the group is a
$|SO(g_t)|$-cycle. Therefore, the permutation induced by $\psi$ on
$D_{t,u}$ is a product of $\ell$ disjoint $k$-cycles, where $k \ell = p
\pm 1$ is even. Thus, the sign of the permutation is $(-1)^{(k+1) \ell}
= (-1)^{k \ell + \ell} = (-1)^\ell$. By
Proposition~\ref{prop:quadchar}, this is equal to the quadratic
character of $v$ modulo $p$.
\end{proof}

\section{Cycle structure and sign of permutation}

In this section, we analyze the cycle structure of the permutation
$\tau_Q$ and show that its sign is $\legendre{q}{p}$. Once again the
analysis splits up into various cases.

\medskip
\noindent \emph{Case} 1A: $b = c = d = 0$.
Here $\tau_Q$ is the identity permutation, so it has $p+1$ fixed
points. The sign of the permutation is $1$, which equals $\legendre{q}{p}$
since $q$ is a square mod~$p$, because $\bar{q} = a^2$.

\medskip
\noindent \emph{Case} 1B: $\bar{q} = a^2$, but $(b,c,d) \neq (0,0,0)$.
In this case, $\tau_Q$ has exactly one fixed point on the conic $C_p$,
so there remain $p$ points which it must permute. Note that $N = \phi_Q
- I$ has eigenvalue zero with multiplicity three, so it is nilpotent:
in fact, by the Cayley-Hamilton theorem, $N^3 = 0$, and since $p \geq
3$, we have $N^p = 0$. Then $\phi_Q^p = (I + N)^p = I + N^p = I$, since
we are in characteristic $p$. Therefore the orbit sizes must all divide
$p$. This forces the remaining $p$ points to be in one orbit, since $p$
is prime. The permutation is thus a $p$-cycle, with sign $(-1)^{p+1} =
1$, which is again the same as $\legendre{q}{p}$.

For the remaining cases, we make the following observations. Let $v_0 =
(b,c,d)$ be an eigenvector with eigenvalue $1$ and nonzero norm. If $v
= (x,y,z)$ is a vector on the cone $K_p$ and orthogonal to $v_0$, then
\begin{align*}
bx + cy + dz &= 0, \\
x^2 + y^2 + z^2 &= 0.
\end{align*}
As we saw in the analysis of Case~2 in Section~\ref{sec:fixed}, this
system has zero or two solutions up to scaling. In the latter case,
both are eigenvectors for $\phi_Q$ with eigenvalue $-1$.  These fixed
points contribute nothing to the sign of the permutation, so it
suffices to look at the orbits of vectors that are not orthogonal to
$v_0$.

Now if $v$ is a vector on the cone $K_p$ not orthogonal to $v_0$
(there are $p\pm 1$ of these up to scaling), then we may fix a scaling
for $v$, by replacing it by ${v}/{\langle v, v_0 \rangle}$. Thus we
may assume that all these $v$ satisfy $\langle v, v_0 \rangle =
1$. The linear transformation $\phi_Q$ permutes these vectors, for if
$\langle v_0, v \rangle = 1$, then
\[
\langle \phi_Q(v), v_0 \rangle  = \langle \phi_Q(v), \phi_Q(v_0) \rangle = \langle v, v_0 \rangle = 1.
\]
For any such $v$, its projection along $v_0$ is the vector $w_0 =
v_0/\langle v_0, v_0 \rangle$, and its projection orthogonal to $v_0$
is $v_\perp = v - w_0$. Note that
\[
\langle v_\perp, v_\perp \rangle = \langle v, v \rangle  - \langle w_0 , w_0 \rangle = 0 - \langle w_0, w_0 \rangle = -\frac{1}{\langle v_0, v_0 \rangle}.
\]

Conversely, given any vector $w$ of norm $-1/\langle v_0, v_0 \rangle$
orthogonal to $v$, there is a unique vector $v$ on the cone $K_p$ with
$\langle v, v_0 \rangle = 1$ and $v^\perp = w$, namely $v = w + w_0$.

Because $v_0$ is fixed, $\phi_Q$ induces an orthogonal transformation
on the orthogonal complement of $v_0$. The quadratic form on this
subspace has discriminant $1/\langle v_0, v_0 \rangle$ and must
therefore be equivalent to the diagonal form $f_Q(x,y) = x^2 - ty^2$,
where $t = -\langle v_0, v_0 \rangle^{-1} = -(b^2 + c^2 + d^2)^{-1} =
(a^2 - \bar{q})^{-1} $.

Therefore, $\phi_Q$ gives rise to an element $\psi_Q \in SO_2(f_Q)$,
whose characteristic polynomial is
\[
x^2  + 2\left(1 - \frac{2a^2}{\bar{q}}\right) x + 1
\]
by Lemma~\ref{lemma:charpoly}. To understand the orbits of $\phi_Q$ on
the projectivization of the cone $K_p$, it is enough to understand the
action of $\psi_Q$ on the vectors of norm $u = -1/\langle v_0, v_0
\rangle$ in $\F_p^2$.

\medskip
\noindent \emph{Case} 2: $a = 0$. In this case, the
map $\psi_Q$ is simply multiplication by $-1$, so the permutation
breaks up into $\frac{1}{2} \left(p - \legendre{-q}{p}\right)$ disjoint
transpositions. Because
\begin{align*}
\frac{1}{2} \left(p - \legendre{-q}{p}\right) &= \frac{1}{2}\left(p -
  \legendre{-1}{p}\right) + \frac{1}{2}\legendre{-1}{p} \left( 1 -
  \legendre{q}{p} \right) \\
& \equiv \frac{1}{2}\left( 1 - \legendre{q}{p} \right) \pmod 2,
\end{align*}
we see that the sign of the permutation is equal to $\legendre{q}{p}$.

\medskip
\noindent \emph{Case} 3: $a \neq 0$ and $\bar{q} \neq a^2$.
We apply Corollary~\ref{cor:sign}, with $t = a^2 - \bar{q}$ and $v =
4a^2/\bar{q}$, to see that the sign of the permutation is the quadratic
character of $4a^2/\bar{q}$ modulo $p$, i.e., $\legendre{q}{p}$.

We have thus proved the following theorem.

\begin{theorem}
Let $p$ and $q$ be distinct primes, with $p$ odd. The sign of the permutation
induced on the $p + 1$ Hurwitz primes of norm $p$ (up to left
multiplication by units) through metacommutation by a prime of norm $q$
is $\legendre{q}{p}$.
\end{theorem}

This completes the proof of Theorem~\ref{theorem:main}.  Note
that the cycle structure of the permutation is somewhat more
subtle: it is determined by $a$, $\bar{q}$, and the order of
the roots of $x^2 + 2(1-2a^2/\bar{q})x+1$ in the multiplicative
group $\F_{p^2}^\times$.

\section*{Acknowledgements}

The authors thank Noam Elkies, Benedict Gross, Bjorn Poonen, and the anonymous referee
for helpful comments.

\end{document}